\documentclass{article}
\usepackage[usenames, dvipsnames]{xcolor}
\usepackage[hyphens]{xurl}
\usepackage[colorlinks=true, breaklinks=true, linkcolor={Blue}, urlcolor={Blue}, citecolor={Red}]{hyperref}
\usepackage{tikz}
\usetikzlibrary{shapes, arrows, positioning, quotes}
\usepackage{blkarray}
\newcommand{\defin}[1]{{\emph{#1}}\index{#1}}
\usepackage[english]{babel}
\addto\extrasenglish{}
\addto\extrasenglish{}
\usepackage{mathtools}
\usepackage{sectsty}
\usepackage{float}
\usepackage{subfig}
\usepackage{caption}
\usepackage{titlesec}
\usepackage{todonotes}

\sectionfont{\fontsize{15}{17}\selectfont}
\subsectionfont{\fontsize{15}{17}\selectfont}

\usepackage[a4paper, margin=2.5cm]{geometry}
\usepackage{csquotes}

\usepackage{enumitem}
\MakeOuterQuote{"}
\usepackage[backend=biber,style=numeric, 
maxbibnames=20, 
url=false, 
isbn=false, 
doi=false]{biblatex}
\DeclareLabelalphaTemplate{
  \labelelement{
    \field[final]{shorthand}
    \field{label}
    \field[strwidth=3,strside=left,ifnames=1]{labelname}
    \field[strwidth=1,strside=left]{labelname}
  }
  \labelelement{
    \field{year}    
  }
}

\newbibmacro{string+doiurlisbn}[1]{%
  \iffieldundef{doi}{%
    \iffieldundef{url}{%
      \iffieldundef{isbn}{%
        \iffieldundef{issn}{%
          #1%
        }{%
          \href{http://books.google.com/books?vid=ISSN\thefield{issn}}{#1}%
        }%
      }{%
        \href{http://books.google.com/books?vid=ISBN\thefield{isbn}}{#1}%
      }%
    }{%
      \href{\thefield{url}}{#1}%
    }%
  }{%
    \href{http://dx.doi.org/\thefield{doi}}{#1}%
  }%
}

\DeclareFieldFormat{title}{%
  \usebibmacro{string+doiurlisbn}{\mkbibemph{#1}}}
\DeclareFieldFormat
  [article,inbook,incollection,inproceedings,patent,thesis,unpublished]
  {title}{\usebibmacro{string+doiurlisbn}{\mkbibquote{#1\isdot}}}
\DeclareFieldFormat
  [suppbook,suppcollection,suppperiodical]
  {title}{\usebibmacro{string+doiurlisbn}{#1}}

\usepackage{authblk}

\usepackage{amsmath,amssymb,amsthm}
\usepackage{thmtools}
\usepackage{thm-restate}

\declaretheorem[name=Lemma, numberwithin = section]{lemma}
\declaretheorem[name=Theorem,sibling = lemma]{theorem}

\declaretheorem[style=definition,name=Definition, sibling=lemma]{definition}
\declaretheorem[name=Notation, sibling=lemma]{notation}
\declaretheorem[name=Corollary, sibling=lemma]{corollary}

\declaretheorem[name=Problem]{problem}

\declaretheorem[name=Example]{example}
\declaretheorem[style=definition,name=Remark, sibling=lemma]{remark}

\usepackage[noabbrev,capitalise,nameinlink]{cleveref}
\crefname{theorem}{Theorem}{Theorems}
\crefname{lemma}{Lemma}{Lemmas}
\crefname{claim}{Claim}{Claims}
\crefname{subclaim}{Sub-Claim}{Sub-Claims}
\crefname{observation}{Observation}{Observations}
\crefname{remark}{Remark}{Remarks}
\crefname{corollary}{Corollary}{Corollaries}
\crefname{definition}{Definition}{Definitions}
\crefname{notation}{Notation}{Notations}
\crefname{assumption}{Assumption}{Assumptions}
\crefname{conjecture}{Conjecture}{Conjectures}
\crefname{question}{Question}{Questions}
\crefformat{equation}{(#2#1#3)}
\Crefformat{equation}{Equation #2(#1)#3}
\newenvironment{txteq}
{
	\begin{equation}
	\begin{minipage}[t]{0.85\textwidth} 
	\em                                
}
{\end{minipage}\end{equation}\ignorespacesafterend}
\newenvironment{txteq*}
{
	\begin{equation*}
	\begin{minipage}[t]{0.85\textwidth} 
	\em                                
}

\newcommand{\rededge}{\mathord{\,\color{red}\vrule width 3pt height 3pt depth -1.5pt\,\vrule width 3pt height 3pt depth -1.5pt\,\vrule width 3pt height 3pt depth -1.5pt}\,}
\newcommand{\blueedge}{\mathord{\,\color{blue}\vrule width 12pt height 3pt depth -1.5pt}\,}
\newcommand{\edge}{\mathord{\,\color{black}\vrule width 2.5pt height 3pt depth -1.5pt\,\vrule width 2.5pt height 3pt depth -1.5pt\,\vrule width 2.5pt height 3pt depth -1.5pt}\,}

\newcommand{\cay}{\mathop{\mathsf{Cay}}}
\newcommand{\Z}{\mathbb{Z}}
\newcommand{\dd}{\mathsf{deg}}
\newcommand{\sm}{\smallsetminus}

\usepackage[classfont=sanserif, langfont=roman,funcfont=italic]{complexity}

\title{\LARGE{On Prime Matrix Product Factorizations}}

\author{Saieed Akbari
\\
\vspace{-0.4cm}
Department of Mathematics,\\ Sharif University of Technology, Tehran, Iran\\
\vspace{0.5cm}
Mohamad Parsa Elahimanesh\\
Department of Mathematics,\\ Sharif University of Technology, Tehran, Iran\\

\vspace{0.5cm}
Bobby Miraftab\\
School of Computer Science, \\Carleton University, Ottawa, Ontario, Canada\\

}
\date{\today}

\begin{document}

\maketitle

\fontsize{12}{16}\selectfont

\begin{abstract}
A graph $G$ factors into graphs $H$ and $K$ via a matrix product if $A = BC$, where $A$, $B$, and $C$ are the adjacency matrices of $G$, $H$, and $K$, respectively.  The graph $G$ is prime if, in every such factorization, one of the factors is a perfect matching that is, it corresponds to a permutation matrix.
We characterize all prime graphs, then using this result we classify all factorable forests, answering a question of Akbari et al. [\emph{Linear Algebra and its Applications} (2025)]. 
We prove that every torus is factorable, and we characterize all possible factorizations of grids, addressing two questions posed by Maghsoudi et al. [\emph{Journal of Algebraic Combinatorics} (2025)].
\end{abstract}

\section{Introduction}

We say that \( G \) is factored into \( H \) and \( K \) if, for some vertex orderings, 
their adjacency matrices \( A, B, \) and \( C \) satisfy \( A = BC \).
The graph \( G \) is called \emph{factorable} if such a factorization exists.
Let $H$ and $K$ be two graphs with $V(K)=V(H)=[n]=\{1,\ldots,n\}$.
If $B$ and $C$ are the adjacency matrices of $H$ and $K$, respectively, then the \defin{matrix product} $HK$ is a weighted digraph with adjacency matrix $BC$.

\noindent 
In this paper, we investigate the factorization problem for the ``matrix product of graphs''.
Recent papers \cite{spectral,herman2025matrix,maghsoudi2023matrix,Manjunatha,inf_fac} renewed interest in matrix-product factorizations.
Specifically, Maghsoudi et al. posed the following question:

\begin{problem}{\rm \cite[Problem 1]{maghsoudi2023matrix}}
Which class of graphs is factorable?
\end{problem}

Consider $K_{2,2}$, where vertices 1 and 2 are adjacent to vertices 3 and 4. Then, one can see that
$$
\begin{bmatrix}
0 & 0 & 1 & 1 \\
0 & 0 & 1 & 1 \\
1 & 1 & 0 & 0 \\
1 & 1 & 0 & 0
\end{bmatrix}=
\begin{bmatrix}
0 & 1 & 0 & 0 \\
1 & 0 & 0 & 0 \\
0 & 0 & 0 & 1 \\
0 & 0 & 1 & 0
\end{bmatrix}
\begin{bmatrix}
0 & 0 & 1 & 1 \\
0 & 0 & 1 & 1 \\
1 & 1 & 0 & 0 \\
1 & 1 & 0 & 0
\end{bmatrix}
$$
We call such a graph \defin{prime}. More precisely, we say a graph $G$
is \emph{prime} if one of the factors in every factorization of $G$ is a perfect matching.
In this paper, we study prime graphs. 
In fact, we develop some tools and show that there is a characterization of prime graphs in terms of the existence of a matched pair.

\begin{definition}\label{matched_pair}
Let $H$ and $K$ be graphs on the same vertex set. A pair $(u,v)$ is called a \defin{matched pair} if one of the following holds:
\begin{enumerate}
    \item $\dd_H(u) = \dd_H(v) = 1$ and $v \in N_H(u)$.
    \item $\dd_K(u) = \dd_K(v) = 1$ and $v \in N_K(u)$.
\end{enumerate}

\noindent Without loss of generality, assume $\dd_H(u) = \dd_H(v) = 1$ and $v \in N_H(u)$.

\end{definition}

\begin{restatable}{theorem}{mainnn}
\label{self-match}
The graph $G$ is prime if and only if every factorization contains a matched pair endpoint in each component of $G$.
\end{restatable}

\noindent As a consequence of studying prime graphs, we introduce new classes of factorable graphs.
For instance, we characterize all factorable forests which answers a question in \cite{spectral}.
Let $G$ be a forest and $\mathcal I$ be the set of all isolated vertices of $G$.
$G$ is factorable if and only if $G\sm \mathcal I$ is factorable.
Our first application is the following:
\begin{restatable}{theorem}{mainn}
\label{mainn2}
If a forest $G$, with no isolated vertices, is factorable, then the following holds:

\begin{enumerate}[label=\rm{(\roman*)}]
    \item For every component $C$ of $G$, there is another component of $G$, say $C'$, such that $C \cong C'$.
    \item The union  $C \cup C'$ can be factored into a perfect matching and itself.

\end{enumerate}
\end{restatable}

\begin{example}
In the following example, we see a factorization of a disjoint union of two copies of path $P_8$.
\end{example}
\begin{figure}[H]
    \centering
    \includegraphics[scale=0.7]{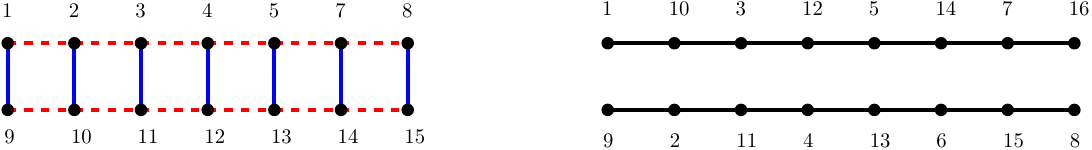}
    \caption{$P_8 \cup P_8$ is factored into $P_8 \cup P_8$ and a perfect matching.}
\end{figure}

\noindent In addition to our first application, we characterize all factorable grid graphs (see \cref{grid}) and torus graphs (see \cref{torus}), thereby resolving two questions posed by Maghsoudi \emph{et al.}~\cite{maghsoudi2023matrix}.

\begin{restatable}{theorem}{gridsss}
\label{gridss}
Let \( m, n \in \mathbb{N} \).
Then the Cartesian product \( P_n \square P_m \) is factorable if and only if both \( m \) and \( n \) are even.
\end{restatable}

\begin{example}
In the following example, we factor the grid $P_2 \square P_4$ into a perfect matching and some graph.
\end{example}
\begin{figure}[H]
    \centering
    \subfloat[\centering ]{{\includegraphics[scale=0.7]{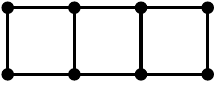} }}%
    \qquad
    \subfloat[\centering]{{\includegraphics[scale=0.7]{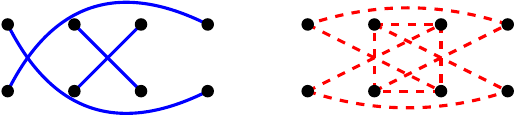} }}%
    \caption{The left graph depicts $P_2 \square P_4$, which is factored into the two graphs on the right.}%
\end{figure}

\begin{restatable}{theorem}{torusss}
\label{toruss}
Let \( m, n \in \mathbb{N} \).
Then $C_m \square C_n$ is factorable.
\end{restatable}

\begin{example}
In the following example, we factor $C_3 \square C_3$ into two 2-regular graphs.
\end{example}
\begin{figure}[H]
    \centering
    \subfloat[\centering ]{{\includegraphics[scale=0.7]{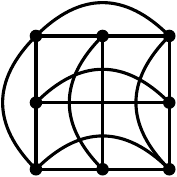} }}%
    \qquad
    \subfloat[\centering]{{\includegraphics[scale=0.7]{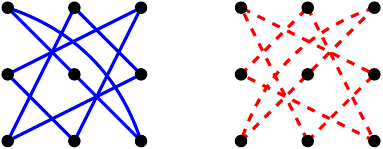} }}%
    \caption{Factorization of $C_3\square C_3$ into two 2-regular graphs.}%
\end{figure}

\section{Preliminaries}
Throughout, $G$ is a \emph{simple} graph (no loops or parallel edges) without isolated vertices, with vertex set $V(G)$ and edge set $E(G)$.
Let $u$ and $v$ be two vertices in $G$. If $u$ and $v$ are adjacent, then we denote the edge between them as $u\edge v$.

\noindent In \cite{maghsoudi2023matrix}, Maghsoudi et al. presented an equivalent definition of the matrix product in a purely combinatorial manner. First, we need to establish a few notations and definitions.

\begin{notation}{\rm\cite{Rosen}}
  Let $H$ and $K$ be graphs with the same vertex set $V(H) = V(K) = [n]$.  
  The disjoint union of $H$ and $K$, denoted by $H \oplus K$, is the graph with vertex set $[n]$ and edge multiset $E(H) \sqcup E(K)$, where edges occurring in both $H$ and $K$ appear with multiplicity.
\end{notation}

\noindent Let $H$ and $K$ be two graphs with $V(H)=V(K)=[n]$.
A path $u\edge w\edge v$ is a $(H-K)$-path of length $2$ if $u\edge w\in E(H)$ and $w\edge v\in E(K)$.
(The endpoints $u$ and $v$ may coincide.)

\begin{definition}
Let $H$ and $K$ be two graphs with $V(H)=V(K)=[n]$.
Then $HK$ is a digraph(multi-arcs and loops are allowed) on $[n]$ such that
\begin{txteq}
    $i\xrightarrow{k}j$ if and only if  $\exists $ $k$ $(H-K)$-paths of length $2$ from $i$ to $j$ in $H\oplus K$
\end{txteq}
The notation $i\xrightarrow{k}j$ means there are $k$ arcs from $i$ to $j$. 
\end{definition}

In this paper, we seek graphs $H$ and $K$ with the adjacency matrices $B$ and $C$ such that $BC$ is a $(0,1)$-symmetric matrix with zero diagonal.

We next describe a combinatorial characterization of factorable graphs.
Recall that a $k$-edge coloring of a graph assigns one of $k$ colors to each edge. 
This coloring need not be a proper edge coloring.

\begin{definition}{\rm(cf. \cite[Definition 2.6]{godsil2001algebraic})}
We say that $H \oplus K$ satisfies the \defin{diamond condition} if, for every pair of vertices $u,v$, whenever there exists an $(H-K)$-path of length $2$ from $u$ to $v$, there exists a unique $(K-H)$-path of length $2$ from $u$ to $v$.
\end{definition}

\begin{figure}[H]
    \centering
    \subfloat[$a \rededge b \blueedge d$ is a $(K-H)$-path and $a \blueedge c \rededge d$ is an $(H-K)$-path.]{{\includegraphics[scale=0.65]{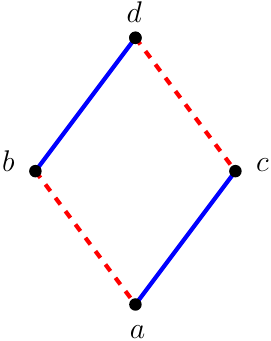} }}%
    \qquad
    \subfloat[$K_5$ satisfies the diamond condition: whenever there is an $(H-K)$-path $u \blueedge w \rededge v$, there is a unique $(K-H)$-path $u \rededge x \blueedge v$ too.]{{\includegraphics[scale=1.1]{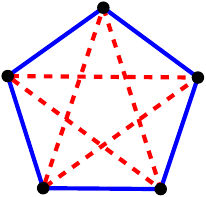} }}%
    \caption{}
    \label{dimond}
\end{figure}

\noindent Color the edges of $H$ blue and those of $K$ red; their union $H\oplus K$ is thus a 2-edge-colored graph.  Understanding this coloring is key to characterizing when $G=HK$ exists.

\begin{lemma}{\rm\cite[Theorem 1]{Manjunatha}}\label{dia_con}
Let a graph $G$ be factored into graphs $H$ and $K$.
Then the following holds:
\begin{enumerate}[label=\rm{(\roman*)}]
    \item  $H\oplus K$ is simple i.e. $H$ and $K$ are edge disjoint.
    \item There is a $2$-edge coloring of $H\oplus K$ in which $E(H)$ is blue and $E(K)$ is red, satisfying the diamond condition.
\end{enumerate}  
\end{lemma}

\begin{lemma}{\rm\cite[Corollary 2]{Manjunatha}}\label{deg}
Let $G$ be factored into graphs $H$ and $K$.
If vertices $u$ and $v$ are connected in $H$, then $\dd_K(u)={\dd}_K(v)$.
\end{lemma}

\begin{lemma}{\rm\cite[Theorem 7]{Manjunatha}}\label{thm7}
Let $G$ be factored into graphs $H$ and $K$ and $u$ is a vertex of $G$.
Then $\dd_G(u) = \dd_H(u)\dd_K(u)$.
\end{lemma}
\section{Prime Graphs}

We start with the following simple lemma without proof.

\begin{lemma}\label{B2=1}
Let $H$ be a perfect matching with the adjacency matrix $B$. Then $B^2=I$.
\end{lemma}

\noindent Any adjacency matrix of a perfect matching is a \emph{permutation matrix}, which means that in every row and every column there is exactly one 1. In addition, every $n \times n$ permutation matrix corresponds to a permutation on $S_n$, which is a product of disjoint transpositions. Each transposition corresponds to an edge of the matching. Thus, we have the following corollary:

\begin{corollary}\label{aut_H}
Let $H$ be a perfect matching with the adjacency matrix $B$.
Then $B$ can be considered as an automorphism of $H$.
\end{corollary}

\noindent Let $G$, $H$, and $K$ be graphs with the adjacency matrices $A$, $B$, and $C$ respectively, such that $A = BC$ and $B^2 = I$.

Since $V(H) = V(G)$, then $B$ can be considered as a map from $V(G)$ to $V(G)$.
We next show that $B$ also can be regarded as an automorphism of $G$.
But we need the following lemma.

\begin{lemma}\label{ABA}
Let $H$ be a perfect matching with the adjacency matrix $B$. 
If $G$ is a graph with the adjacency matrix $A$, then $BAB$ represents a relabeling of $G$.
\end{lemma}

\begin{proof}
By \Cref{B2=1}, $B^2 = I$, so $B = B^{-1}$. Since $H$ is a perfect matching, $B$ is a permutation matrix. Thus $BAB = BAB^{-1} = BAB^T$, which means that $BAB$ relabels $G$ by swapping each vertex with its matched partner in $H$.
\end{proof}

\noindent Every automorphism of a graph $G$ can be regarded as a permutation matrix. Keeping this fact in mind, we prove the following theorem.

\begin{theorem}\label{B_aut_G}
Let $G$ be a graph and $H$ be a perfect matching with the adjacency matrix $B$.
Then $G$ is factored into $H$ and some $K$ if and only if $B$ is an automorphism of $G$ without any fixed edge.
\end{theorem}

\begin{proof}
Let $A$ and $C$ be the adjacency matrices of graphs $G$ and $K$ so that $A = BC = CB$.
Since $B^2 = I$, $BA = AB = C$. Thus $B$ is an automorphism of $G$.
For the sake of contradiction, suppose $B$ fixes some edge $\{u,v\}$ of $G$ (so $Bu = v$ and $Bv = u$). Then $\{u,v\}$ is an edge of both $G$ and $H$, contradicting \Cref{dia_con} (which states that $G$ and $H$ have no edge in common). 

Conversely, assume $B$ is an automorphism of $G$ with no fixed edge. Then $BA = AB$. Because $B$ and $A$ are symmetric matrices, $BA$ is also symmetric. Moreover, since $B$ fixes no edge of $G$, every diagonal entry of $BA$ is $0$ (no edge of $G$ appears in both $G$ and $H$). Thus $BA$ is a symmetric $\{0,1\}$-matrix with zero diagonal, which we take as the adjacency matrix of a graph $K$. By construction $A = B(BA) = B C$, so indeed $G$ is factored into $H$ and $K$.
\end{proof}

\begin{notation}
Let $G$ be a graph and $u$ be a vertex of $G$. Then $C_G(u)$ denotes the component of $G$ containing $u$.
\end{notation}

\begin{definition}  
  Let $G$ and $H$ be graphs with $U \subseteq V(G)$.  
  A map $\varphi\colon U \to V(H)$ is a \defin{graph homomorphism} if $u \edge v \in E(G)$ implies $\varphi(u) \edge \varphi(v) \in E(H)$ for all $u, v \in U$.  
\end{definition}  

\begin{lemma}  
  Let $G$ be factored into $H$ and $K$ with a matched pair $(u,v)$.  
  For $s \in \{u,v\}$, the map $\varphi\colon C_K(s) \to N_H(C_K(s))$, defined by $\varphi(x) = N_H(x)$, is a graph homomorphism.  
\end{lemma}

\begin{proof}  
  By \Cref{deg}, every $x \in C_K(u) \cup C_K(v)$ has a unique neighbor in $H$, so $\varphi(x) = N_H(x)$ is well-defined.
  In particular, $\varphi(u) = v$ and $\varphi(v) = u$. To prove that $\varphi$ is a graph homomorphism, let $a, b \in C_K(s)$ be adjacent in $K$. 
  Since each vertex has a unique neighbor in $H$, we obtain an alternating path $N_H(a) \blueedge a \rededge b \blueedge N_H(b)$. By the diamond condition, this implies $N_H(a) \rededge N_H(b)$, proving that $\varphi(a) \edge \varphi(b) \in E(N_H(C_K(s)))$.
\end{proof}  

\begin{corollary}  
We have $C_K(u) = \varphi(C_K(v))$ and $\varphi(C_K(u)) = C_K(v)$.  
\end{corollary}  
  
\begin{proof}  
Since $\varphi(C_K(s)) \subseteq C_K(\varphi(s))$, applying $\varphi$ again gives $\varphi^2(C_K(s)) \subseteq \varphi(C_K(\varphi(s)))$.  
By the definition of $\varphi$, we have $\varphi^2(C_K(s)) = C_K(s)$, which implies $\varphi(C_K(s)) = C_K(\varphi(s))$.  
Thus, $N_H(C_K(s)) = C_K(N_H(s))$.  
\end{proof}

\begin{corollary}\label{aut}
$\varphi\colon C_K(u) \to C_K(v)$ is an isomorphism, and moreover $N_H(C_K(u))=C_K(v)$ and $N_H(C_K(v))=C_K(u)$.
\end{corollary}

\begin{definition}
A subset $U$ of $V(G) = V(H \oplus K)$ is called \defin{alone} if every vertex $u'$ of $V(G) \sm U$ has no neighbors in $U$ in  $H \oplus K$ and $G$, i.e. $N_{G}(u')\cap U=N_{H\oplus K}(u')\cap U=\emptyset $.
\end{definition}

\begin{lemma}\label{alone}
Let a graph $G$ be factored into $H$ and $K$ with a matched pair $(u,v)$.
Then $C_K(u)\cup C_K(v)$ is alone.
\end{lemma}

\begin{proof}
According to the diamond condition, every edge of $G$ corresponds to an alternating path in $H \oplus K$. This implies that every component of $G$ is a subset of a component of $H \oplus K$. Therefore, it is sufficient to show that $C_K(u) \cup C_K(v)$ is alone in $H \oplus K$.
Let $x \in V(G) \sm V(C_K(u) \cup C_K(v))$. We need to show that $N_{H \oplus K}(x) \cap (C_K(u) \cup C_K(v)) = \emptyset$.
By \cref{aut}, we know that $N_H(C_K(u)) = C_K(v)$ and $N_H(C_K(v)) = C_K(u)$. Therefore, $x$ cannot be adjacent to any vertex in $C_K(u) \cup C_K(v)$ in $H$.
Furthermore, $C_K(u)$ and $C_K(v)$ are components of $K$, and thus there is no edge in $K$ between $x$ and $C_K(u) \cup C_K(v)$.
\end{proof}

\begin{definition}
Let $G$ be a graph. Then $G$ is \defin{prime} if, for every factorization of $G$, one of the factors is a perfect matching.
\end{definition}
\noindent By definition, A graph without any factorization is prime.

\begin{remark}\label{even}
If a prime graph $G$ is factorable, then $G$ has an even number of vertices.
\end{remark}

\noindent We now are ready to prove the main theorem of this section.

\mainnn*

\begin{proof}
If $G$ factors into $H$ and $K$ with $H$ a perfect matching, then trivially the endpoints of each $H$-edge form a matched pair (as desired).\\
For the backward implication, assume that $G$ factors into $H$ and $K$. Consider any vertex $x$ in $H$, and let $C$ be the component of $G$ containing $x$. By assumption, there is a matched pair ${u, v}$ such that $u \in C$.
By \Cref{alone}, $C_K(u) \cup C_K(v)$ is alone, which implies that $x \in C_K(u) \cup C_K(v)$.
By \Cref{deg}, $\dd_H(z) = 1$ for every $z \in C_K(u) \cup C_K(v)$, it follows that $\dd_H(x) = 1$, which implies that $H$ is a perfect matching.
\end{proof}


\section{Applications of Prime Graphs}

\noindent We now apply our prime-graph results to broader classes: first graphs without $C_4$ (yielding factorable forests), then grids and torus graphs.

\subsection{Square-free graphs}
In this subsection, we characterize all factorable forests.
We start with the following lemma.

\begin{lemma}\label{degree=1}{\rm{\cite[Lemma 19]{spectral}}}
Let a graph $G$ be factored into two graphs $H$ and $K$. If $G$ has no $C_4$, then for every vertex $u \in V(G)$, either $\dd_K(u) = 1$ or $\dd_H(u) = 1$.
\end{lemma}

\begin{lemma}\label{no-c4}  
  Every graph with no $C_4$ is prime.
\end{lemma}  
  
\begin{proof}
Let $C$ be a component of $H \oplus K$.
If $C$ is 2-regular, then for each vertex $u\in V(C)$ we have $\deg_H(u)=\deg_K(u)=1$, so each $H$-edge in $C$ is a \emph{matched pair}.
If $C$ is not 2-regular, then there exists a vertex $u \in C$ with $\deg_{H\oplus K}(u)\ge 3$. By \Cref{degree=1}, we may assume w.l.o.g. that $\deg_H(u)=1$ and $\deg_K(u)\ge 2$. 
Then $u$ has a unique neighbor $N_H(u)$ in $H$, and by \Cref{deg} we get $\deg_K(N_H(u)) = \deg_K(u) \ge 2$. 
Applying \Cref{degree=1} again, we find $\deg_H(N_H(u)) = 1$, so $\{u,\,N_H(u)\}$ is a matched pair. Since $C$ was an arbitrary component of $H \oplus K$, \Cref{self-match} ensures that $G$ is prime.  
\end{proof}

\begin{lemma}\label{fix_uni_}
Let $G$ be a graph such that for a pair of vertices $u$ and $v$ there is a unique shortest path $\gamma$ between them.
Then there is no perfect matching $H$ with an adjacency matrix $B$ such that $B$ is an automorphism of $G$ without fixed edge and $Bu=v$.   
\end{lemma}

\begin{proof}
Suppose for contradiction that there exists a perfect matching $H$ such that $B u = v$. If $\gamma$ has an odd number of vertices, then $B$ fixes the central vertex of $\gamma$. This is because there is only one path between $u$ and $v$, and so $B \gamma = \gamma$. However, $B$ cannot fix a vertex of $G$, since $H$ is a perfect matching. Therefore, $\gamma$ has an even number of vertices and so $B$ fixes the central edge of $\gamma$, contradicting the assumption that $B$ fixes no edge of $G$.
\end{proof}

\noindent Now, we have the following immediately corollary.

\begin{corollary}\label{autT}
For every tree $T$ and every perfect matching $H$ with adjacency matrix $B$,
if $B$ is an automorphism of $T$, then $B$ fixes an edge of $T$.
\end{corollary}

\noindent In \cite[Problem 2]{spectral}, the authors ask for a characterization of factorable forests. We answer this question as follows.

\mainn*

\begin{proof}
Let $G$ be a forest factored into $H$ and $K$. Since $G$ has no $C_4$, \Cref{no-c4} implies $G$ is prime; therefore $H$ is a perfect matching. Let $B$ be the adjacency matrices of $H$. By \cref{autT}, $B$ permutes the components of $G$ (it does not fix any component).

Next, let $ X $ be the adjacency matrix of a component. Since $G$ is bipartite, $X$ has the following form:
\[
\begin{pmatrix}
0 & Y \\
Y^T & 0
\end{pmatrix}.
\]
Now, consider two isomorphic components of $ G $, whose adjacency matrix is of the form:
\[
\begin{pmatrix}
0 & Y & 0 & 0 \\
Y^T & 0 & 0 & 0 \\
0 & 0 & 0 & Y \\
0 & 0 & Y^T & 0
\end{pmatrix}.
\]
If this matrix is multiplied by the adjacency matrix of a perfect matching:
\[
\begin{pmatrix}
0 & 0 & I & 0 \\
0 & 0 & 0 & I \\
I & 0 & 0 & 0 \\
0 & I & 0 & 0
\end{pmatrix},
\]
the result is,
\[
\begin{pmatrix}
0 & 0 & 0 & Y \\
0 & 0 & Y^T & 0 \\
0 & Y & 0 & 0 \\
Y^T & 0 & 0 & 0
\end{pmatrix},
\] the proof is complete.
\end{proof}

\begin{corollary}{\rm\cite[Theorem 20]{spectral}}\label{tree}
Every tree with at least $2$ vertices is not factorable. 
\end{corollary}


\subsection{Degree Sequences}

In this subsection, we study the interaction between the degree sequence of a graph and its primeness.
We start with the following lemma.
\begin{lemma}\label{abcd}
  Let $G$ be a factorable graph. If two vertices $a$ and $d$ are adjacent in $G$, then there exist adjacent vertices $b$ and $c$, such that $a$, $b$, $c$, and $d$ are pairwise distinct and
  $$
  \dd_G(a)\, \dd_G(d) = \dd_G(b)\, \dd_G(c).
  $$
\end{lemma}

\begin{proof}
Let $G$ be factored into $H$ and $K$. Since $a$ and $d$ are adjacent in $G$, there exist vertices $b$ and $c$ satisfying the diamond condition (see \cref{dimond}(a)). By \cref{thm7}, we have:
\begin{align*}
\dd_G(a)\, \dd_G(d) &= \dd_K(a)\, \dd_H(a)\, \dd_K(d)\, \dd_H(d)\\
&= \dd_K(c)\, \dd_H(b)\, \dd_K(b)\, \dd_H(c)\\
&= \dd_G(c)\, \dd_G(b).\qedhere
\end{align*}
\end{proof}

\noindent As a direct consequence of the preceding result, we have the following corollary.
\begin{corollary}
Let $G$ be a connected graph with two adjacent vertices of degrees $p$ and $q$, where $p$ and $q$ are prime numbers. If $G$ has \emph{neither}  
\begin{enumerate}[label=\rm{(\roman*)}]
  \item a vertex of degree $pq$, nor
  \item any vertex of degree 1,  
\end{enumerate}
then $G$ is prime.

\end{corollary}

\begin{proof}
Let $\dd_G(a) = p$ and $\dd_G(d) = q$ and $G$ be factored into $K$ and $H$. By \Cref{abcd} and without loss of generality, we can assume that there are two other vertices $b$ and $c$ such that $\dd_G(b) = p$ and $\dd_G(c) = q$. 
Note $a, b, c$, and $d$ should satisfy the diamond condition (see \cref{dimond}(a)).
Given that $\dd_G(c) = q$, so either $\dd_H(c) = q$ or $\dd_K(c) = q$, see \cref{thm7}. 
If $\dd_K(c) = q$, then $\dd_K(a) = q$, as $c$ is adjacent to $a$ in $H$, see \cref{deg}.
However, this yields a contradiction, as $p=\dd_G(a) =\dd_H(a)\dd_K(a)=q \dd_H(a)$. 
Thus, we have $\dd_H(c) = q$ and so $\dd_K(c)=1$.
Note $d$ and $c$ are adjacent in $K$ and so $\dd_H(c)=\dd_H(d)=q$.
Since $\dd_G(d)=\dd_H(d)\dd_K(d)$, we have $\dd_K(d)=1$.
Since $c$ and $d$ are adjacent in $K$ and $\dd_K(d)=\dd_K(c)=1$, they form a matched pair.
\end{proof}

\begin{lemma}\label{pk}
Let $G$ be a connected graph with a vertex of degree $p$ and no vertices of degree $kp$, where $p$ is a prime number and $k > 1$. 
Then $G$ is prime.
\end{lemma}

\begin{proof}
Let $G$ be factored into $H$ and $K$, and let $a$ be a vertex of degree $p$ in $G$. 
There are vertices  $b,c$ and $d$ satisfying the diamond condition (see \cref{dimond}(a)).
Since $\dd_G(a) = p$, without loss of generality, we can assume that $\dd_H(a) = 1$ and $\dd_K(a) = p$. 
Note $a$ is adjacent to $c$ in $K$, so $\dd_K(a) = \dd_K(c) = p$. Since $\dd_G(c) = \dd_H(c) \dd_K(c)$, we have $\dd_G(c) = p \, \dd_H(c)$, which implies that $\dd_H(c) = 1$.
Given that $\dd_H(c) = \dd_H(a) = 1$ and $c$ is adjacent to $a$ in $H$, we deduce that ${c,a}$ is a matched pair. Finally, since $G$ is connected \Cref{self-match} completes the proof.
\end{proof}

\begin{corollary}\label{p_prime}
  Let $G$ be a connected graph whose maximum degree is prime. Then $G$ is prime.
\end{corollary}
\noindent Next, we show that the Petersen graph is not factorable.
\begin{corollary}
The Petersen graph is not factorable.
\end{corollary}
\begin{proof}
Let us denote the Petersen graph by $P$. 
It is cubic, and by \cref{p_prime}, It is prime.
\begin{figure}[H]
    \centering
    \includegraphics[scale=0.8]{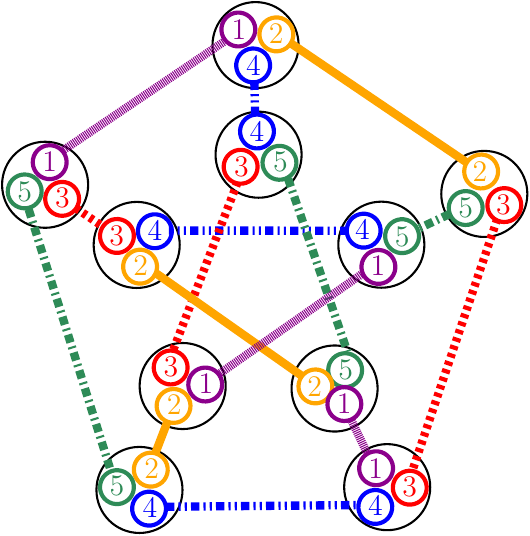}
    \label{Petersen graph}
\end{figure}

\noindent

Next, we show that every automorphism of $P$ of order two without a fixed edge has a fixed vertex. This implies that $P$ has no factorization; see \Cref{B_aut_G}. It is known that the automorphism group of $P$ is the symmetric group $S_5$, see \cite{MR3552765}.
Let $e$ be an arbitrary edge of $P$. Without loss of generality, we can assume that $e$ is colored number $1$, as $P$ is edge-transitive. 
Assume that $\sigma$ is an automorphism of $P$ of order 2 that maps the number 1 to 2, so there are only four cases for $\sigma$:
$$(1,2),\,(1,2)(3,4),\,(1,2)(4,5), \text{ and }(1,2)(3,5)$$
In all cases, $\sigma$ must fix the unique vertex incident with the edges colored $3$ (red), $4$ (blue), and $5$ (green); hence $\sigma$ has a fixed vertex.\qedhere
 
\end{proof}

\subsection{Factorization of Grids and Torus}

\noindent Maghsoudi et al.\ \cite{maghsoudi2023matrix} asked to characterize all factorable grid and torus graphs. We answer these questions here.

\begin{definition}{\rm(cf. \cite[p. 382]{domi})}\label{grid}
A \defin{two-dimensional grid} graph, also known as a two-dimensional lattice graph is an $m\times n$ lattice graph that is the Cartesian product $P_m  \square  P_n$ of path graphs on $m$ and $n$ vertices. 
\end{definition}

\begin{lemma}\label{grid_prime}
Every grid graph is prime.
\end{lemma}

\begin{proof}
Let $G=P_n\square P_m$.
We split our proof depending on the values of $m$ and $n$.
\begin{itemize}[leftmargin=*]
    \item If $n=1$ and $m\geq 2$, then $G \cong P_m$ and the result follows from \Cref{tree}.
    \item If $n=m=2$, then $G \cong C_4$.
    All vertices have degree 2 and none has degree $2k\,(k>1)$, so \Cref{pk} applies.  

    \item Let $n\geq 2$ and $m\geq 3$. Here $G$ contains vertices of degree 3 but none of degree $3k\,(k>1)$; \Cref{pk} again yields the result.
    \qedhere
\end{itemize}
\end{proof}

Let $G$ be a connected graph. For $u,v\in V(G)$, let $d(u,v)$ denote their (shortest-path) distance.  
For $v\in V(G)$, its \defin{eccentricity} is $\epsilon(v)=\max_{u\in V(G)} d(v,u)$.  
The \defin{radius} of $G$ is $r=\min_{v\in V(G)} \epsilon(v)$, and the \defin{center} is
$$
\mathrm{cen}(G)=\{v\in V(G)\mid \epsilon(v)=r\}.
$$
      
\begin{remark}\label{fix_center}
Every automorphism of $G$ fixes $\mathrm{cen}(G)$ setwise.
\end{remark}
\noindent Next, we have the following folklore result.
\begin{lemma}\label{center_grid}
Let $m, n \in \mathbb{N}$. 
\begin{enumerate}[label=\rm{(\roman*)}]
    \item If exactly one of $m$ or $n$ is even, then the center of $P_n \square P_m$ has exactly two adjacent vertices.
    \item If both $m$ and $n$ are even, then the center of $P_n \square P_m$ has exactly four vertices.\qedhere
\end{enumerate}
\end{lemma}
\noindent In the next theorem, we discuss when $P_n \square P_m$ is factorable. If $m = n = 1$, then $P_1 \square P_1$ is a single vertex, and its adjacency matrix is a $1 \times 1$ zero matrix.
Trivially, every $n \times n$ zero matrix is the matrix product of two $n \times n$ zero matrices.

\gridsss*
\begin{proof}
Let \( G = P_n \square P_m \).  

First, assume $G$ is factorable. By \Cref{grid_prime}, $G$ is prime, and by \Cref{even} $G$ has an even number of vertices. Hence at least one of $m$ or $n$ is even.

Suppose for contradiction that exactly one of $m,n$ is even. Then center of $G$ consists of exactly two adjacent vertices (by \Cref{center_grid}). By \Cref{fix_center}, every automorphism of $G$ fixes this center edge. However, by \Cref{B_aut_G}, \( B \) is an automorphism of \( G \) with no fixed edges, contradicting the previous observation. Thus, both \( m \) and \( n \) must be even.  

For the backward implication, assume that both \( m \) and \( n \) are even.  
The vertex set of \( G \) is  
\[
V(G) = \{(i,j) \mid 1 \leq i \leq n, \ 1 \leq j \leq m \}.
\]  
There is a unique shortest path between \( (1,1) \) and \( (1,m) \) and similarly between \( (1,1) \) and \( (n,1) \).  
Since \( B \) is an automorphism of \( G \), it follows that  
\[
B(1,1) = (n,m), \quad B(n,1) = (1,m).
\]  
Because \( B \) does not fix any vertex of \( G \), we deduce that  
\[
B(1,i) = (n, m+1-i), \quad B(j,1) = (n+1-j, m).
\]  
Applying the same argument,  
\[
B(2,i) = (n-1, m+1-i), \quad B(j,2) = (n+1-j, m-1).
\]

By induction, $B(i,j) = (n+1-i, m+1-j)$ defines the unique automorphism, proving factorability.
\end{proof}

\begin{definition}{\rm(cf. \cite[p. 382]{domi})}\label{torus}
The \defin{torus} or torus grid graph $T_{m, n}$ is the graph formed from the Cartesian product $C_m \square C_n$ of the cycle graphs $C_m$ and $C_n$. 
\end{definition}

\begin{definition} \label{CayleyDefn}
{\rm(cf. \cite[p. 34]{godsil2001algebraic})} If $S$ is a symmetric subset(A set~$S$ is \emph{symmetric} if $S = S^{-1}$, where  $S^{-1} = \{\, s^{-1}\mid s\in S \,\}$.) of a group~$G$, then the corresponding \emph{Cayley
graph} $\cay(G;S)$ is the undirected graph whose vertices are the elements of $G$, and such that vertices
$x$ and $y$ are adjacent if and only if $x^{-1}y \in S$.
\end{definition}

\noindent The torus $C_n\square C_m$ is isomorphic to $\cay(\mathbb Z_n\times \mathbb Z_m;(\pm 1,0),(0,\pm 1))$.

\torusss*

\begin{proof}
Let us denote the vertices of $C_m \square C_n$ by 
$$V(C_n\square C_m)=\{(i,j)\mid 0\leq j\leq m-1,0\leq i\leq n-1\}.$$
Depending on $m$ and $n$ we have the following cases:a symmetric subset (A set~$S$ is \emph{symmetric} if $S = S^{-1}$, where  $S^{-1} = \{\, s^{-1}\mid s\in S \,\}$.)
\begin{itemize}[leftmargin=*]
    \item Let at least one of $m$ and $n$ be even. Without loss of generality, assume that $n$ is even.
    We define the graph $H$ with the same vertex set $V(C_n \square C_m)$ and the edge set $\{(i,j) \blueedge (i+\frac{n}{2},j) \mid 0\leq i \leq \frac{n}{2}, 0\leq j \leq m-1\}$. 
    Note that the sum in component occurs modulo $\mathbb{Z}_n$ and $\mathbb{Z}_m$ respectively.
    It is not hard to see that $H$ is a perfect matching on $V(C_n \square C_m)$. Now, consider the adjacency matrix $B$ of $H$. $B$ has no fixed edge, and by \cref{B_aut_G}, $G$ is factored into $H$ and some graph $K$.
    
    \item Let $m$ and $n$ be both odd. In this case, we define $H$ and $K$ directly. 
    We set $H\coloneqq \cay(\Z_n\times \Z_m,(\frac{n+1}{2},\frac{m-1}{2}),(\frac{n-1}{2},\frac{m+1}{2}))$ and $K\coloneqq \cay(\Z_n\times \Z_m,(\frac{n-1}{2},\frac{m-1}{2}),(\frac{n+1}{2},\frac{m+1}{2}))$.
    Assume that $H$ has blue edges and $K$ red edges(dashed edges).
    Let $ u\blueedge v\rededge w$ and $u\rededge z \blueedge t$, where $u=(i,j)$. Then we have the following cases:
    \begin{itemize}
        \item We first assume that $ u+(\frac{n+1}{2},\frac{m-1}{2})=v$ and $v+(\frac{n-1}{2},\frac{m-1}{2})=w$.
        Thus we have $ u+(n,m-1)=w$ and so $w=(i,j-1)$.
        \item We assume that $ u+(\frac{n-1}{2},\frac{m+1}{2})=v$ and $v+(\frac{n-1}{2},\frac{m-1}{2})=w$.
        Thus we have $ u+(n-1,m)=w$ and so $w=(i-1,j)$.
        \item We next assume that $ u+(\frac{n+1}{2},\frac{m-1}{2})=v$ and $v+(\frac{n+1}{2},\frac{m+1}{2})=w$.
        Thus we have $ u+(n+1,m)=w$ and so $w=(i+1,j)$.
        \item We next assume that $ u+(\frac{n-1}{2},\frac{m+1}{2})=v$ and $v+(\frac{n+1}{2},\frac{m+1}{2})=w$.
        Thus we have $ u+(n,m+1)=w$ and so $w=(i,j+1)$.
    \end{itemize}
    Hence we prove that $(i,j)$ is adjacent to $(i-1,j),(i,j-1),(i+1,j)$ and $(i,j+1)$.
    Next, we need to show that $H\oplus K$ has the diamond condition.
    Assume that the vertices $(i,j)$ and $(i',j')$ belong to the following diamond:
    $$(i,j)\blueedge X_1\rededge (i',j')\blueedge X_2 \rededge (i,j),$$
where $X_i \in V(H \oplus K)$ for $i = 1, 2$. Note $(i', j')$ is one of the following cases: $(i-1, j)$, $(i, j-1)$, $(i+1, j)$, and $(i, j+1)$. In addition, one can see that the vertices $X_1$ and $X_2$ are defined uniquely for each case.
This shows that the matrix product of $H$ and $K$ is $C_n\Box C_m$.\qedhere
\end{itemize}
\end{proof}

\section{Further research}
We close the paper with the following problems:
\begin{enumerate}
\item Let $d$ and $t$ be integers with $d \ge 1$ and $t \ge 3$. The
\defin{windmill graph $W_{d,t}$} is obtained by taking $d$ copies of the
clique $K_t$ and identifying exactly one vertex from each copy into a
single vertex. We call this shared vertex $c$ and refer to it as the
\defin{hub}.
Each copy contributes $t-1$ new vertices, so there are $d(t-1)$ non-hub
vertices in total. Furthermore, each non-hub vertex lies in exactly one
copy of $K_t$ and is adjacent to the hub $c$ and to the other $t-2$
vertices of its own clique.
For example, when $t = 3$, the graph $W_{d,3}$ is exactly the standard
friendship graph $F_d$, and it follows from \cite[Proposition 3.29]{maghsoudi2023matrix}
that $F_d$ (and hence $W_{d,3}$) does not admit a factorization.
On the other hand, the windmill graph $W_{1,4t+1} = K_{4t+1}$ does admit
a factorization.

We ask to determine all pairs $(d,t)$ for which $W_{d,t}$ admits
a factorization.

\item The authors in \cite{maghsoudi2023matrix} proved that the complete graph $K_n$ admits a factorization if and only if $n=4t+1$.
They also provide two factorizations of $K_{4t+1}$, where both factors are transitive in both cases.
We ask the following: find two non-transitive graphs $H$ and $K$ with adjacency matrices $B$ and $C$ respectively such that $BC=J-I$ 
\end{enumerate}

\printbibliography

\end{document}